\documentclass[a4paper, 12pt]{article}

\usepackage[top=3.4cm, left=3.4cm, right=3.4cm,bottom=3.4cm]{geometry}

\usepackage{graphicx}
\usepackage{amsmath, amssymb, amsthm}
\def \e {\varepsilon}

\usepackage{cases}

\newcommand{\R}{{\mathbb R}}
\newcommand{\N}{{\mathbb N}}

\newtheorem{theorem}{Theorem}[section]

\newtheorem{lemma}[theorem]{Lemma}
\newtheorem{definition}[theorem]{Definition}
\newtheorem{defi}[theorem]{Definition}

\newtheorem{corollary}[theorem]{Corollary}
\newtheorem{proposition}[theorem]{Proposition}

\theoremstyle{remark}
\newtheorem{remark}[theorem]{Remark}

\usepackage{enumitem}

\DeclareMathOperator{\esssup}{ess \, sup}
\DeclareMathOperator{\essinf}{ess \, inf}

\DeclareMathOperator{\ess}{ess}

\newcommand{\mintO}{\displaystyle {\int \kern -0.961em -}_{\Omega}}
\newcommand{\mintOs}{\displaystyle {\int \kern -0.961em -}_{\Omega^\sharp}}

\usepackage{color}
\definecolor{darkblue}{rgb}{0,0,0.7}

\begin{document}

\title{On the large time behavior of the solutions of a nonlocal
ordinary differential equation
with mass conservation}
\author{Danielle Hilhorst\thanks{\texttt{CNRS, Laboratoire de Math\'{e}matique, Analyse Num\'{e}rique et EDP, Universit\'e de Paris-Sud, F-91405 Orsay Cedex, France}},
 Hiroshi Matano\thanks{\texttt{Graduate School of Mathematical Sciences, University of Tokyo, Komaba, Tokyo 153-8914, Japan}},\\
Thanh Nam Nguyen\thanks{\texttt{Laboratoire de Math\'{e}matique, Analyse Num\'{e}rique et EDP, Universit\'e de Paris-Sud, F-91405 Orsay Cedex, France}}\, and Hendrik Weber \thanks{\texttt{Mathematics Institute, University of Warwick, Coventry CV4 7AL, United Kingdom}}
}
\date{}
\maketitle

\medskip

\noindent{\bf Abstract.}
We consider an initial value problem for a nonlocal differential
equation with a bistable nonlinearity in several space dimensions.
The equation is an ordinary differential equation
with respect to the time variable $t$, while the nonlocal term
is expressed in terms of spatial integration.  We discuss the
large time behavior of solutions and prove, among other things,
the convergence to steady-states. The proof that the solution
orbits are relatively compact is based upon the rearrangement theory.

\date{}          


\section{Introduction}


\subsection{Motivation}

In this paper, we study solutions $u(x,t)$ of
the initial value problem
\begin{equation*}
 (P) \ \ \left\{
\begin{array}{ll}
u_t=f(u)-\langle f(u) \rangle \quad &  x\in \Omega, \, t >0,\vspace{6pt}\\
u(x,0)=u_0(x)  & x\in \Omega,
\end{array}
\right.
\end{equation*}
where $\Omega$ is a bounded open set in $\R^N$ with $N \ge 1$ and
$$
\langle f(u) \rangle:=\frac{1}{|\Omega|}\int_{\Omega} f(u(x,t))\,dx.
$$
Here
$|A|$ denotes the Lebesgue measure of a set $A\subset \Omega$,
$u_0$ is a bounded function on $\Omega$ and $f(u)$ is a bistable
nonlinearity. More precise conditions on $f$ will be specified
later. A typical example is $f(u)=u-u^3$. Note that the Problem
$(P)$ does not change if we replace $f$ by $f+c$ where $c$ is any
constant.

Solutions of Problem $(P)$ satisfy the following mass conservation
property
\begin{equation}\label{mass:conservation:modify}
\int_\Omega u(x,t)\,dx=\int_\Omega u_0(x) \,dx \mbox{~~for all~~}t
\ge 0.
\end{equation}
This is easily seen by integrating the equation in $(P)$, but we
will state it more precisely in Theorem \ref{principal:theorem}.
Note also that Problem $(P)$ formally represents a gradient flow
for the functional
\begin{equation}\label{3:11:2014:Hamlyapunove}
E(u)=-\int_{\Omega} F(u) \mbox{~~with~~}F(s)=
\int_0^s f(\sigma)
\,d \sigma
\end{equation}
on the space
$X:=\{w \in L^2(\Omega): \int_\Omega w \,dx=\int_\Omega
u_0\,dx\}$
(see Remark \ref{gradientlfow:chobaitaon}).

Let us explain the motivation of the present work briefly.
The singular limit of reaction-diffusion equations, in particular,
the generation of interface and the motion of interface have been
studied by many authors, for instance \cite{AlfaroHilhorstMatano},
\cite{Chen}, \cite{de-motton}. This paper is motivated by the
study of a generation of interface property for the Allen-Cahn
equation with mass conservation:
\begin{equation*} \label{Allen-Cahn:motivation:equation}
 (P^\e) \ \  \left\{
\begin{array}{ll}
u^\varepsilon_t=\Delta u^\varepsilon+\dfrac{1}{\varepsilon^2}
(f(u^\varepsilon)-\langle f(u^\varepsilon)\rangle)
& \mbox{in } \Omega\times (0,\infty),\vspace{5pt}\\
\dfrac{\partial u^\e}{\partial \nu}=0 & \mbox{on } \partial\Omega
\times (0,\infty), \vspace{8pt}\\
u^\e(x,0)=u_0(x)  & x\in \Omega.
\end{array}
\right.
\end{equation*}
We refer the reader to \cite{Rubinstein} for the motivation of
studying this model.

If we use a new time scale $\tau := t/\varepsilon^2$ in order to
study the behavior of the solution at a very early stage, then the
equation in $(P^\e)$ is rewritten as follows:
\begin{equation}\label{rescale}
u^\varepsilon_\tau=\e^2 \Delta u^\varepsilon+f(u^\varepsilon)-
\langle f(u^\varepsilon)\rangle.
\end{equation}
If the term $\e^2 \Delta u^\varepsilon$ is negligible for sufficiently
small $\e$, then \eqref{rescale} reduces to the equation in $(P)$,
thus one can expect that $(P)$ is a good approximation of $(P^\e)$
in the time scale $\tau$ so long as $\e^2 \Delta u^\varepsilon$ is
negligible.

In the standard Allen-Cahn problem
$u^\varepsilon_t=\Delta u^\varepsilon+\dfrac{1}{\varepsilon^2}
f(u^\varepsilon)$,
it is known that the term $\e^2 \Delta u^\varepsilon$ is indeed
negligible for a rather long time span of $\tau=O(|\ln\e|)$, or
equivalently $t=O(\e^2|\ln\e|)$.
During this period, the solution typically develops steep transition
layers (generation of interface).
After that,
the term $\e^2 \Delta u^\varepsilon$ is no longer negligible;
see, for example, \cite{AlfaroHilhorstMatano, Chen} for details.

It turns out that the same is true of Problem $(P^\e)$, as we
will show in our forthcoming paper \cite{HMNW}. In other words,
$(P^\e)$ is well approximated by the ordinary differential
equation $(P)$ for a rather long time span in $\tau$.  Thus,
in order to analyze the behavior of solutions of $(P^\e)$ at the
very early stage, it is important to understand the large time
behavior of solutions of $(P)$.

We will show in this paper that the solution of $(P)$ converges
to a stationary solution as $t \to\infty$.  Furthermore, in
many cases the limit stationary solutions are step functions
that take two values $a_-$ and $a_+$ which satisfy
$$f(a_-)=f(a_+), \quad f'(a_-) < 0, \quad f'(a_+) < 0.$$
Consequently, one may expect that, at the very early stage, the
solution $u^\e$ to Problem $(P^\e)$ typically approaches a
configuration consisting of plateaus near the values $a_-$ and
$a_+$ and steep interfaces (transition layers) connecting the
regions
$\{x\in \Omega: u^\e \approx a_- \}$ and
$\{x\in \Omega: u^\e \approx a_+\}$.
In order to make
this conjecture rigorous, one needs a more precise error estimate
between $(P)$ and $(P^\e)$, which will be given in the
above-mentioned forthcoming article \cite{HMNW}.

\subsection{Main results}
Before stating our main results, let us specify our assumptions on
the nonlinearity $f$. For the most part we will consider a bistable
nonlinearity satisfying $\bf (F_1), (F_2)$ below,
but some of our results hold under a milder condition $\bf (F')$,
which allows $f$ to be a multi-stable or a more general nonlinearity.

\begin{itemize}
\item[$\bf( F_1)$]  $f \in C^1(\R)$, and there exist $m<M$ satisfying
$f'(m)=f'(M)=0$ and
$$f'<0 \mbox{~~on~~}(-\infty,m) \cup (M,\infty), \quad  f'>0
\mbox{~~on~~}(m,M).$$
\item[$\bf( F_2)$] There exist $s_*<s^*$ satisfying
\begin{equation}\label{chap2:Definition:s:star:va:s:sup:star}
\begin{cases}
s_*<m<M< s^*, \\
f(s_*)=f(M), \quad f(s^*)=f(m).
\end{cases}
\end{equation}
\end{itemize}

\begin{figure}[!h]
\begin{center}
\includegraphics[scale=0.58]{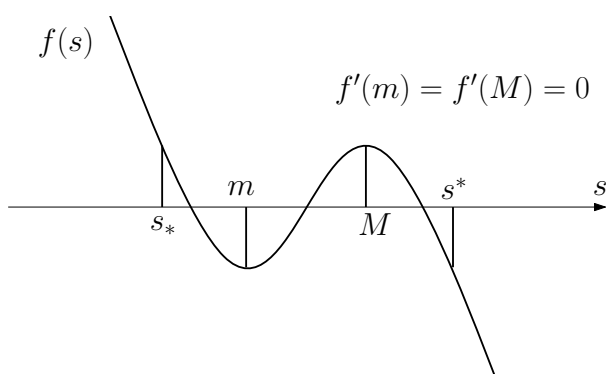}
\caption{{Bistable nonlinearity}}
\label{default}
\end{center}
\end{figure}

A milder version of our assumption is the following.
This assumption allows $f$ to be a multistable nonlinearity.
Any periodic nonlinearity such as $f(s)=\sin s$ also satisfies
this milder assumption.

\begin{itemize}
\item[$\bf( F')$] $f: \R \to \R$ is locally Lipschitz continuous
and there exist two sequences $\{a_n\}$, $\{b_n\}$, with
$a_n \le b_n$,  such that $a_n \to -\infty$,  $b_n \to \infty$ as
$n \to \infty$ and that
$$f(a_n) \ge  f(s) \ge f(b_n) \mbox{~~for all~~} s \in [a_n, b_n].$$
\end{itemize}

Clearly $\bf (F_1)$, $\bf (F_2)$ imply $\bf (F')$. The assumption
$\bf (F')$ will be used only in Theorem \ref{theorem:multi-stable:case}
 below.
Unless otherwise stated, we will always assume $\bf (F_1), (F_2)$.

In Theorems \ref{principal:theorem} and
\ref{principal:theorem:to:kai:po} below we present existence and
uniqueness results for the solution of Problem $(P)$ and give basic
properties of the $\omega$-limit set under the assumption that $u_0 \in L^\infty(\Omega)$.
Note that we will henceforth regard $(P)$ as an ordinary differential
equation on $L^\infty(\Omega)$. More precisely, $u$ is a solution
on the interval $0 \le t < T$ if
$u \in C^1([0, T); L^\infty(\Omega))$ and satisfies
\begin{equation*}
 (P') \ \ \left\{
\begin{array}{ll}
\dfrac{du}{dt}=H(u) &  t >0, \vspace{6pt}\\
u(0)=u_0,
\end{array}
\right.
\end{equation*}
where $H: L^\infty(\Omega) \to L^\infty(\Omega)$ is defined by
\begin{equation}\label{term:nonlocal-nonlinear}
H(u):=f(u)-\langle f(u) \rangle.
\end{equation}
As we will see later in \eqref{problem:ODE:22:6:3:14}, Problem $(P')$ is equivalent (a.e. in $\Omega$)
to the original Problem $(P)$, which is considered pointwise for each $x \in \Omega.$
We will abbreviate $u(x,t)$ as $u(t)$ when we regard $u$
as an $L^\infty(\Omega)$-valued function. We may mix the two notations so long as there is no fear of confusion.

For any $u_0 \in L^\infty(\Omega)$, we fix two constants $s_1<s_2$ such that
$$s_1, s_2 \not\in (s_*, s^*)$$
and that
$$s_1 \le u_0(x) \le s_2 \mbox{~~for a.e.~~}x\in\Omega.$$
For example, if we set
\begin{equation}\label{definetion:esssup,-essinf:ky10:for:u0}
\alpha:=\essinf_{x \in \Omega} u_0(x), \quad \beta:=\esssup_{x \in \Omega} u_0(x),
\end{equation}
then $s_1, s_2$ can be chosen as follows:
\begin{equation*}
\begin{cases}
[\alpha, \beta] \cap (s_*, s^*)=\varnothing  \quad \Rightarrow \quad s_1:=\alpha, s_2:=\beta,\\
[\alpha, \beta] \cap (s_*, s^*) \neq \varnothing  \quad \Rightarrow \quad s_1=\min\{\alpha, s_*\},  s_2=\max\{\beta, s^*\}.
\end{cases}
\end{equation*}
For later reference we
summarize the above hypotheses as follows:

\medskip

\noindent $\bf( H)$
$s_1, s_2\not\in(s_*,s^*)$
\mbox{~~and~~} $s_1 \le u_0(x) \le s_2 \mbox{~~for a.e.~~}x\in\Omega.~$

\medskip

\begin{theorem}\label{principal:theorem}
Assume that $u_0 \in L^\infty(\Omega)$ and let $s_1,s_2$ be as in
$\bf (H)$.
Then Problem $(P)$ possesses a unique time-global solution
$u \in C^1([0,\infty); L^\infty(\Omega))$. Moreover,
the solution has the mass conservation property \eqref{mass:conservation:modify}and, for all $t \ge 0$,
\begin{equation}\label{ine:for:solution:in-thm:s1:s2}
s_1 \le u(x,t) \le s_2 \mbox{~~~for a.e.~~} x\in \Omega.
\end{equation}
\end{theorem}

The next theorem is concerned with the structure of the $\omega$-limit set.
Here the $\omega$-limit set of a solution $u(t)$ of $(P)$ (or more precisely $(P')$) with initial data $u_0$
is defined as follows:

\begin{defi}
We define the $\omega$-limit set of $u_0$ by
$$\omega(u_0):=\{\varphi \in L^1(\Omega): \exists t_n \to \infty, u(t_n) \to \varphi \mbox{~~in~~} L^1(\Omega)\mbox{~as~} n \to \infty\}.$$
\end{defi}

The reason why we use the topology of $L^1(\Omega)$, not $L^\infty(\Omega)$,
to define the $\omega$-limit set is that the solution often develops sharp transition layers
which cannot be captured by the $L^\infty(\Omega)$ topology. Note that, since the solution
$u$ is uniformly bounded in $L^\infty(\Omega)$ by \eqref{ine:for:solution:in-thm:s1:s2},
 the topology of $L^p(\Omega)$ is equivalent to that of
$L^1(\Omega)$ for all $p \in [1, \infty)$.

\begin{theorem}\label{principal:theorem:to:kai:po}
Let $u_0, s_1, s_2$ be as in Theorem \ref{principal:theorem}. Then
\begin{enumerate}[label=\emph{(\roman*)}]
\item $\omega(u_0)$ is a nonempty compact connected set in $L^1(\Omega)$.
\item Any element $\varphi \in \omega(u_0)$ satisfies $s_1 \le \varphi \le  s_2$ a.e. and is a stationary
solution of $(P)$. More precisely, it satisfies
$$
f (\varphi(x)) - \langle f (\varphi) \rangle = 0
\mbox{~~~for a.e.~~} x\in\Omega.
$$
\end{enumerate}
\end{theorem}

Note that the constant $\langle f (\varphi) \rangle$ that appears
in the above equation is not specified {\it a priori}.  It depends
on each $\varphi$. This is in marked contrast with the standard
Allen-Cahn problem, in which stationary solutions are defined
by the equation $f(\varphi)=0.$

The next theorem is a generalization  of Theorems \ref{principal:theorem} and \ref{principal:theorem:to:kai:po} to a larger class of nonlinearities
including multi-stable ones. In this theorem (and only here) we assume $(\bf F')$ instead of $\bf (F_1), (F_2)$.

\begin{theorem}[The case of multi-stable nonlinearity]\label{theorem:multi-stable:case}
Assume $\bf (F')$ and let $u_0 \in L^\infty(\Omega)$.
Choose $n\in\N$ such that $a_n\leq u_0(x)\leq b_n\;a.e.\;x\in\Omega$.
Then Problem $(P)$ possesses a unique time-global solution
$u \in C^1([0, \infty); L^\infty(\Omega))$.
The solution satisfies \eqref{mass:conservation:modify}
and, for each $t \ge 0$,
$$a_n \le u(x,t) \le b_n \mbox{~~for a.e.~~}x\in\Omega.$$
Furthermore, the conclusions of Theorem
\ref{principal:theorem:to:kai:po} hold with $s_1, s_2$
replaced by $a_n, b_n$, respectively.
\end{theorem}

The following two theorems are concerned with the fine structure of the $\omega$-limit sets under the assumptions $\bf (F_1), (F_2)$.

\begin{theorem} \label{principal:theorem2}
Let $u_0 \in L^\infty(\Omega)$ satisfy
$\langle u_0 \rangle \not\in [s_*, s^*]$.
Then $\omega(u_0)$ has a unique element $\varphi$ which is given by
$$\varphi=\langle u_0 \rangle.$$
Moreover, there exist constants $C_1, C_2>0$ such that
\begin{equation}\label{convergecne:rate:to:10000}
\|u(t)-\langle u_0 \rangle\|_{L^\infty(\Omega)} \le C_1 \exp(-C_2t) \mbox{~~for all~~}t \ge 0.
\end{equation}
\end{theorem}

\begin{theorem}\label{hequa:chinh4}
Let $u_0 \in L^\infty(\Omega)$ satisfy
$s_* \le \langle u_0\rangle  \le s^*$.
Assume
that
\begin{equation}\label{giathietchodinhlyso222dg}
|\{x: u_0(x)=s \}|=0 \mbox{~~for all~~}s \in \R.
\end{equation}
Then $\omega(u_0)$ contains
only one element $\varphi$ and it
is a step function that takes at most two values. More precisely one of the following holds:
\begin{enumerate}[label=\emph{(\alph*)}]
\item there exist   $a_-, a_+$ with $s_* < a_-< m, M < a_+ < s^*, f(a_-)=f(a_+)$ and disjoint subsets $\Omega_-, \Omega_+ \subset \Omega$
with $\Omega_- \cup \Omega_+=\Omega$ satisfying
$$\varphi(x)=a_- \chi_{\Omega_-}(x)+a_+ \chi_{\Omega_+}(x) \mbox{~~for a.e.~~}x\in\Omega,$$
where $\chi_A(x)$ denotes the characteristic function of a set $A \subset \Omega$;
\item $\varphi(x) \in \{m, s^*\} \mbox{~~for a.e.~~}x\in\Omega$;
\item $\varphi(x) \in \{s_*, M\} \mbox{~~for a.e.~~}x\in\Omega$.
\end{enumerate}
\end{theorem}

\begin{remark}
If $\omega(u_0)$ contains
only one
element $\varphi$, then for all $p \in [1, \infty)$,
$$u(t) \to \varphi \mbox{~~in~~} L^p(\Omega) \mbox{~~as~~} t \to \infty.$$
\end{remark}

The proof of Theorem \ref{hequa:chinh4} is presented in two steps.
First we consider the special case
where $u_0$ satisfies the following pointwise inequalities:
$$s_* \le u_0(x) \le s^* \mbox{~~for a.e.~~}x\in\Omega,$$
and then prove the theorem for the general case.
The results in the special case are stated in Theorem \ref{principal:theorem3}.

One of the difficulties of Problem $(P)$ is the presence of the nonlocal term $\langle f(u) \rangle$. Because of this term, we do not have the standard comparison principle for solutions
of $(P)$, or $(P')$. Another difficulty is
due to the lack of diffusion term, so that
the solution in general is not smooth at all in space. Furthermore, since the solution can develop many sharp transition layers at least in some cases,
the solution does not necessarily remain bound in $BV(\Omega)$, which makes it difficult to prove the relative compactness of the solution
 even in $L^1(\Omega)$.
This difficulty is overcome by considering the unidimensional equi-measurable rearrangement $u^\sharp$ and showing that
it is a solution of a one-dimensional problem
$(P^\sharp)$ to be introduced in Subsection \ref{chap22:subsec4.2}. Since the orbit $\{u^\sharp(t): t \ge 0\}$ is bounded in $BV(\Omega^\sharp)$, where $\Omega^\sharp:=(0, |\Omega|) \subset \R$, it is relatively compact in $L^1(\Omega^\sharp)$.
We then prove that
the convergence of  $u^\sharp(t_n)$ in  $L^1(\Omega^\sharp)$ is equivalent to the convergence of  $u(t_n)$ in $L^1(\Omega)$.  These two facts establish
the relative compactness of $\{u(t): t \ge 0\}$ in $L^1(\Omega)$.

As we will show in Section \ref{section4ch}, Problem $(P)$ possesses a Lyapunov functional  $E(u)$. This and the relative compactness of the solution orbit
imply that the $\omega$-limit set of the solution is nonempty and consists only of stationary solutions of ${(P)}$. However, another difficulty of Problem $(P)$ is that there are too many (actually a continuum) of stationary solutions. This makes it difficult to find a good characterization of the $\omega$-limit set of solutions.
This difficulty will be overcome by careful observation of the dynamics of solutions near the unstable zero of the function
$f(s) - \langle f(u(\cdot,t))\rangle$.

The paper is organized as follows: In Section 2, we prove the existence and uniqueness of the solution of $(P)$ and prove Theorem \ref{principal:theorem}. In Section \ref{chap22:sec4}, using the monotone rearrangement theory, we introduce the one-dimensional problem $(P^\sharp)$ and study the relation between the behavior of solutions of $(P)$ and that of $(P^\sharp)$, which is an essential ingredient for proving Theorem \ref{principal:theorem:to:kai:po}\,(i). In Section \ref{section4ch}, we prove Theorem \ref{principal:theorem:to:kai:po}\,(ii) by using the Lyapunov functional  $E(u)$. We also prove Theorem \ref{theorem:multi-stable:case} for multi-stable nonlinearity in this section. In Section \ref{section:principal123:theorem2}, we prove  Theorem \ref{principal:theorem2}.
Finally, in Section \ref{section:chungminh:menhde:coban}, we prove Theorem \ref{hequa:chinh4}.


\section{Proof of Theorem \ref{principal:theorem}}\label{chap22:sec2}
In this section we prove the global existence and uniqueness of solutions of $(P)$ (or more precisely $(P')$),
which will be done in two steps. In Subsection \ref{subsection2.1},
we only assume that $f$ is locally
Lipschitz continuous and prove the existence of a local-in-time solution as well as its uniqueness. This result is an immediate consequence of the local Lipschitz continuity of the nonlocal nonlinear term in the space $L^\infty(\Omega)$. Then, in Subsection \ref{subsection2.2}, with an additional assumption,
 we prove the uniform boundedness of the solution, which implies the global existence.
Finally in Subsection \ref{subsec:proof:theorem1:mass:conservation}, we apply the results of Subsections \ref{subsection2.1} and \ref{subsection2.2}  to prove Theorem \ref{principal:theorem}.


\subsection{Local existence and uniqueness}\label{subsection2.1}
In this subsection, we always assume that $f: \R \to \R$ is locally Lipschitz continuous.
\begin{lemma}\label{lem:locally:existence:f:lipschitz}
Assume that $f$ is locally Lipschitz continuous. Then for each $u_0 \in L^\infty(\Omega)$, $(P')$ has a unique local-in-time solution.
\end{lemma}
\begin{proof}
Since $f: \R \to \R$ is locally Lipschitz continuous, the Nemytskii
operator $u \mapsto f(u)$ is clearly locally Lipschitz in
$L^\infty(\Omega)$. Therefore $u \mapsto \langle f(u) \rangle$ is
also locally Lipschitz since $\Omega$ is bounded.
It follows that the map $H: L^\infty(\Omega) \to L^\infty(\Omega)$
defined in \eqref{term:nonlocal-nonlinear} is locally Lipschitz
continuous. Hence the assertion of the lemma follows from the
standard theory of ordinary differential equations.
\end{proof}

Denote by $[0, T_{max}(u_0))$, with $0<T_{max}(u_0) \le \infty$,
the maximal time interval for the existence of the solution of
$(P')$ with initial data $u_0$. Then the following lemma also
follows from the standard theory of ordinary differential equations.

\begin{lemma}  
Let $u \in C^1([0, T_{max}(u_0)); L^\infty(\Omega))$ be the solution of $(P')$ with $u_0 \in L^\infty(\Omega)$. If $T_{max}(u_0) <\infty$, then
$$\displaystyle\limsup_{t \uparrow T_{max}(u_0)} \|u(t)\|_{L^\infty(\Omega)}=\infty.$$
\end{lemma}

\begin{corollary} \label{coro:for:global:existence}
If there exists $C>0$ such that
$$ \|u(t)\|_{L^\infty(\Omega)} \le C \mbox{~~for all~~} t \in [0, T_{max}(u_0)),$$
 then $T_{max}(u_0)=\infty$.
\end{corollary}


\subsection{Uniform bounds and global existence}\label{subsection2.2}

In this subsection, we continue to assume that $f$ is locally Lipschitz continuous. We fix $u_0 \in L^\infty(\Omega)$ arbitrarily, and write $T_{max}$
instead of $T_{max}(u_0)$ for simplicity.
Set
\begin{equation}\label{definition:lamda:t}
\lambda(t)=\langle f(u(t)) \rangle \mbox{~~for all~~} t \in [0, T_{max}),
\end{equation}
and study solutions $Y(t;s)$ of the following auxiliary problem:
\begin{equation}
(ODE) \,\,
\begin{cases}
\dot{Y}= f(Y)-\lambda(t), \quad t > 0,\vspace{6pt}\\
Y(0)=s,
\end{cases}
\end{equation}
where $\dot{Y}:=dY/dt$.   Since $u \in C^1([0, T_{max}); L^\infty(\Omega))$, it is easy to see that there exits a function $u^*: \Omega \times [0, T_{max}) \to \R$ satisfying
for all $t \in [0, T_{max})$
$$u^*(x,t)=u(x,t) \mbox{~~for a.e.~~}x\in\Omega,$$
and
$$\dfrac{\partial u^*}{\partial t}(x,t)=f(u^*(x,t))-\lambda(t) \mbox{~~for a.e.~~}x\in \Omega \mbox{~~and for all~~} t \in [0, T_{max}).$$
Consequently,
for all $t \in [0, T_{max})$,
\begin{equation}\label{problem:ODE:22:6:3:14}
u(x,t)=Y(t;u_0(x)) \mbox{~~for a.e.~~} x \in \Omega.
\end{equation}

The lemma below proves the monotonicity of $Y(t;s)$ in $s$.

\begin{lemma}\label{cor:monotonicity:of:Y}
Let $\widetilde s <s$ and let $0<T<T_{max}$. Assume that Problem
$(ODE)$ with initial conditions $\widetilde s, s$ possesses the
solutions $Y(t;\widetilde s), Y(t;s) \in C^1([0, T])$, respectively.
Then
\begin{equation}\label{ine:monotone:Y}
Y(t;\widetilde s)<Y(t;s) \mbox{~~for all~~} t  \in [0, T].
\end{equation}
\end{lemma}

\begin{proof}
Since $Y(0;\widetilde s)=\widetilde s<s=Y(0;s)$, the assertion follows immediately from the backward uniqueness of solution of $(ODE)$.
\end{proof}

\begin{lemma}\label{lem:ab-bound}
Let $f: \R \to \R$ be a locally Lipschitz function and let $u$ be the solution of $(P')$ with initial data $u_0\in L^{\infty}(\Omega)$. Assume that there exist $a \leq b$ such that \begin{equation}\label{ab-f}
f(a) \geq  f(s) \geq f(b) \ \ \hbox{for} \ \ s\in[a, b]
\end{equation}
and that
\begin{equation}\label{ab-u0}
a \leq u_0(x) \leq b \quad \hbox{for}\ \ a.e. \ \ x\in\Omega.
\end{equation}
Then the solution $u$ exists globally for $t\geq 0$ and it satisfies, for every $t\geq 0$, \begin{equation}\label{ab-u}
a \leq u(x,t) \leq b \quad \hbox{for}\ \ a.e. \ \ x\in\Omega.
\end{equation}
\end{lemma}

\begin{proof} We will show that \eqref{ab-u} holds so long as the
solution $u$ exists.  The global existence will then follow
automatically from the local existence result in Subsection
\ref{subsection2.1}. Let $Y(t,s)$ be the solution of (ODE).
The assumption \eqref{ab-u0} and the monotonicity of $Y(t;s)$
in $s$ imply
\[
Y(t;a)\leq Y(t;u_0(x))=u(x,t) \leq Y(t;b)\quad\ a.e.\ \ x\in\Omega.
\]
Therefore, in order to prove \eqref{ab-u}, it suffices to show that
\begin{equation}\label{ab-Y}
a\leq Y(t;a), \quad\  Y(t;b)\leq b
\end{equation}
so long as the solutions $Y(t;a), Y(t;b)$ both
exist.

We first consider the special case where $f$ satisfies
\begin{equation}\label{f'<0}
f(s) = f(a)  \ \ \hbox{for} \ \ s\in(-\infty,a], \quad f(s) = f(b)  \ \ \hbox{for} \ \ s\in[b, \infty).
\end{equation}
Then we have $f(s)\geq f(\tilde s)$ for any
$s \in \R$ and $\tilde s \ge b$.
It follows that $f(Y(t;u_0))\geq f(Y(t;b))\mbox{~~a.e. in~~}\Omega$
for every $t\geq 0$ such that $Y(t;b)\geq b$. Hence
$\lambda(t) \geq f(Y(t;b))$ for every such $t$. Consequently,
\[
\dot Y(t;b) = f(Y(t;b))-\lambda(t)\leq 0 \ \ \hbox{for every $t$
such that}\ \ Y(t,b)\geq b.
\]
This, together with $Y(0;b)=b$, proves the second inequality of
\eqref{ab-Y}.  The first inequality of \eqref{ab-Y} can be
shown similarly. This establishes \eqref{ab-Y}, hence \eqref{ab-u},
under the additional assumption \eqref{f'<0}.

Next we consider the general case where \eqref{f'<0} does not
necessarily hold. Let $\tilde f: \R\to \R$ be defined by
\begin{equation*}
\tilde f(s)=
\begin{cases}
f(a) &\mbox{~~for all~~} s <a,\\
f(s) &\hbox{~for} \ s\in[a,b],\\
f(b) &\mbox{~~for all~~} s >b,
\end{cases}
\end{equation*}
and let $\tilde u(t)$ be the solution of the following problem:
\[
\left\{
\begin{array}{ll}
\dfrac{d\tilde u}{dt}=\tilde f(\tilde u)-\langle \tilde f(\tilde u)
\rangle &  t >0,\vspace{6pt}\\
\tilde u(0)=u_0.
\end{array}
\right.
\]
Then, by what we have shown above, the solution $\tilde u(t)$
exists globally for $t\geq 0$, and the following inequalities
hold for every $t\geq 0$:
\[
a \leq \tilde u(x,t) \leq b \quad \hbox{for}\ \ a.e. \ \ x\in\Omega.
\]
Since $\tilde f(s)=f(s)$ for all $s\in[a,b]$, $\tilde u$ is also
a solution of Problem $(P')$, hence $\tilde u(x,t)=u(x,t)\;(a.e.\ x\in\Omega,\;t\geq 0)$ by the uniqueness of the solution of $(P')$.
This implies \eqref{ab-u}, and the proof of Lemma \ref{lem:ab-bound}
is complete.
\end{proof}

The following corollary follows from the proof of
Lemma \ref{lem:ab-bound}.

\begin{corollary} \label{he qua:lalllaa:dfes} Assume that all
the hypotheses in Lemma \ref{lem:ab-bound} hold.
Then for all $s \in [a, b]$, Problem $(ODE)$ has a unique solution
$Y(t;s) \in C^1([0, \infty))$ and it satisfies
$$a \le Y(t;s) \le b \mbox{~~for all~~}s\in [a,b], \ t \ge 0.$$
\end{corollary}


\subsection{Proof Theorem \ref{principal:theorem}}
\label{subsec:proof:theorem1:mass:conservation}

\begin{proof}[\bf Proof Theorem \ref{principal:theorem}]
The local existence and uniqueness follow from Lemma 
\ref{lem:locally:existence:f:lipschitz}. 
Let $s_1, s_2$ be as in $\bf (H)$. We have
$$f(s_1)\ge f(s) \ge f(s_2) \mbox{~~for all~~}s \in [s_1, s_2].$$
Set $a:=s_1, b:=s_2$. Then the assertion
\eqref{ine:for:solution:in-thm:s1:s2} follows from
Lemma \ref{lem:ab-bound}. 
This and Corollary \ref{coro:for:global:existence} imply global 
existence. Next we prove the mass conservation
property \eqref{mass:conservation:modify}.
Integrating the equation in Problem $(P')$ from
$0$ to $t$, we obtain
$$u(t)-u(0)=\int_0^t u_t \,d\tau =\int_0^t [ f(u) - \langle f(u)\rangle] \,d\tau,$$
so that
 $$\int_\Omega u(x,t)\,dx-\int_\Omega u_{0}(x)\,dx=\int_0^t \int_{\Omega}[ f(u) - \langle f(u)\rangle ] \,dxd\tau=0.$$
This completes the proof of Theorem \ref{principal:theorem}.
\end{proof}

\begin{corollary} \label{moiquanheproblemP-ODE}
Let $u$ be the solution of Problem $(P')$ with
$u_0 \in L^\infty(\Omega)$ and let $t_1 \ge 0$. Then
\begin{enumerate}[label=\emph{(\roman*)}]
\item if $s_* \le u(x,t_1) \le s^*$ for a.e. $x \in \Omega$,
then for all $t \ge t_1$,
$$s_* \le u(x,t) \le s^* \mbox{~~for a.e.~~}x\in\Omega;$$
\item if  $u(x, t_1) \le  s_*$ (resp.   $u(x, t_1) \ge  s^*$)
for a.e. $x\in\Omega$, then for all $t \ge t_1$,
$$u(x,t) \le s_* \quad (\mbox{resp.~~}u(x,t) \ge s^*)
\mbox{~~for a.e.~~} x\in \Omega.$$
\end{enumerate}
\end{corollary}

\begin{proof} For (i), we simply set $s_1=s_*, s_2=s^*$ and apply
Theorem \ref{principal:theorem}. For (ii) we set $s_2=s_*$
(resp. $s_1=s^*$).
\end{proof}


\section{Proof of Theorem \ref{principal:theorem:to:kai:po}\,(i)}
\label{chap22:sec4}

This section is organized as follows: We first introduce the rearrangement theory in Subsection \ref{chap22:subsec4.1}. Then in Subsection \ref{chap22:subsec4.2}, we introduce Problem $(P^\sharp)$ and study basic properties of its solutions which we state in Theorem \ref{thm:chobaitaonP:sharp}. Finally, in Section \ref{moiquanhe:sectionP-P:sharp},  we prove the equivalence between the convergence of the solution $u(t)$ and its rearrangement $u^\sharp(t)$. This implies that the solution orbit $\{u(t): t \ge 0\}$ of $(P)$
is relatively compact in $L^1(\Omega)$, from which
Theorem \ref{principal:theorem:to:kai:po}\,(i) follows immediately.


\subsection{Rearrangement theory}
\label{chap22:subsec4.1}

Given a measurable function $w: \Omega \to \R$, the distribution
function of $w$ is defined by
$$\mu_w(\tau):=|\{x \in \Omega: w(x)>\tau\}|, \quad \tau \in \R.$$
Set
$$\Omega^\sharp:=(0, |\Omega|) \subset \R.$$
The \textit{unidimensional decreasing rearrangement} of $w$,
denoted $w^\sharp$, is defined on
$\overline{\Omega^\sharp}=[0, |\Omega|]$ by
\begin{equation}\label{definition:w:sharp}
\begin{cases}
&w^\sharp(0):=\esssup (w), \vspace{6pt}\\
&w^\sharp(y)=\inf\{ \tau:   \mu_w(\tau )<y\}, \quad y>0.
\end{cases}
\end{equation}
We recall some properties of $w^\sharp$  which are stated in
\cite[Chapter $1$]{chap2kesavan}.

\begin{proposition}\label{chap2:mancyuon:ws:ss}
The following properties hold:
\begin{enumerate}[label=\emph{(\roman*)}]
\item $w^\sharp: \Omega^\sharp \to \R$ is nonincreasing and left-continuous.
\item $w, w^\sharp$ are equi-measurable, i.e. $\mu_w=\mu_{w^\sharp}$.
\item If  $G$ is a Borel measurable function such that either $G\ge 0$ or $G(w) \in L^1(\Omega)$ then
$$\int_{\Omega^\sharp} G(w^\sharp(y))\,dy=\int_{\Omega} G(w(x))\,dx.$$
\item Let $w_1, w_2 \in L^p(\Omega)$ with $1\le p \le \infty$; then
$$\|w^\sharp_1-w^\sharp_2\|_{L^p(\Omega^\sharp)} \le \|w_1-w_2\|_{L^p(\Omega)}.$$
\item Let $\Phi:\R \to \R$ be nondecreasing. Then
$$(\Phi (w))^\sharp(y)=\Phi(w^\sharp(y)) \quad\mbox{~~ for a.e.~~} y\in\Omega^\sharp.$$
\end{enumerate}
\end{proposition}

\begin{remark}\label{chuy:tinhchat:w:shap} The function $w^\sharp$ is uniquely defined by the distribution function $\mu_w$. Moreover,
 by Proposition \ref{chap2:mancyuon:ws:ss}\,(ii), if
$$s_1 \le w(x) \le s_2 \mbox{~~for a.e.~~}x\in\Omega,$$
  then
$$s_1 \le  w^\sharp(y) \le s_2 \mbox{~~for all~~}y\in\Omega^\sharp.$$
\end{remark}


\subsection{Associated one-dimensional problem $(P^\sharp)$}
\label{chap22:subsec4.2}

\begin{theorem}\label{thm:chobaitaonP:sharp}
Let $u(t)$ be the solution of Problem $(P')$ with
$u_0 \in L^\infty(\Omega)$ and let $s_1,s_2$ be as in
$\bf (H)$. We define
\begin{equation}\label{defition:u:sharp}
u^\sharp(y,t):=(u(t))^\sharp(y) \mbox{~~on~~} \Omega^\sharp \times [0,\infty).
\end{equation}
Then
 $u^\sharp$ is the unique solution in $C^1([0,\infty); L^\infty(\Omega^\sharp))$ of Problem $(P^\sharp)$:
\begin{equation*}
 (P^\sharp) \ \ \left\{
\begin{array}{ll}
\dfrac{dv}{dt}=f(v)-\langle f(v) \rangle,  \quad  t >0, \vspace{6pt}\\
v(0)=u_0^\sharp. 
\end{array}
\right.
\end{equation*}
Moreover, for all $t \ge 0$,
$$s_1 \le u^\sharp(y,t) \le s_2 \mbox{~~for all~~} y\in\Omega^\sharp,$$
and
\begin{equation}\label{chap3:6:09:bdt:cho:nghiem:u:sao:ha:ha}
u^\sharp(y,t)=Y(t;u_0^\sharp(y)) \mbox{~~for a.e.~~}y\in  \Omega^\sharp.
\end{equation}
\end{theorem}

\begin{proof}
In view of \eqref{problem:ODE:22:6:3:14}, we have
for all $t \ge 0$,
$$u(x,t)=Y(t;u_0(x))  \mbox{~~for a.e.~~} x \in \Omega.$$
Hence, since $Y(t;s)$ is increasing in $s$ (cf. Lemma \ref{cor:monotonicity:of:Y}),  it follows from Proposition \ref{chap2:mancyuon:ws:ss}\,(v) that for all $t\ge 0$,
\begin{equation}\label{chap3:dssfsfs}
u^\sharp(t)=Y(t;u^\sharp_0) \mbox{~~a.e. in~~} \Omega^\sharp.
\end{equation}
Moreover, \eqref{ine:for:solution:in-thm:s1:s2} and Remark \ref{chuy:tinhchat:w:shap} imply that
\begin{equation}\label{chap3:dssfsfs:so2}
s_1 \le u^\sharp(y,t) \le s_2 \mbox{~~for all~~} y\in\Omega^\sharp, t \ge 0.
\end{equation}
It remains to prove that $u^\sharp$ is the unique solution of Problem $(P^\sharp)$. Since $Y(t;s)$ is the solution of Problem $(ODE)$,
 it follows that for all $t\ge 0$,
\begin{align*}
\frac{\partial Y}{\partial t}(t; u^\sharp_0)&=f(Y(t;u^\sharp_0))-\lambda(t)\\
&=f(u^\sharp(t))-\langle f(u(t)) \rangle  \\
&=f(u^\sharp(t))-\langle f(u^\sharp(t)) \rangle \quad \mbox{~~a.e. in~~} \Omega^\sharp.
\end{align*}
Here we have used \eqref{chap3:dssfsfs} and Proposition \ref{chap2:mancyuon:ws:ss}\,(iii).
Since the right-hand-side of the above equation is continuous in $t$ as an $L^\infty(\Omega)$-valued function,
 we have $u^\sharp=Y(\cdot,u^\sharp_0(\cdot)) \in C^1([0,\infty); L^\infty(\Omega^\sharp))$.
Hence $u^\sharp$ is a solution of Problem $(P^\sharp)$  in $C^1([0, \infty); L^\infty(\Omega^\sharp))$.
The uniqueness of the solution of $(P^\sharp)$ follows from the locally Lipschitz continuity of the map $v \mapsto f(v)-\langle f(v) \rangle$ on
$L^\infty(\Omega)$.
\end{proof}

\begin{remark}
The above theorem shows that the monotone rearrangement $u^\sharp(t)$ satisfies precisely the same equation as $(P')$.
The advantage of considering $(P^\sharp)$ is that
its solutions are easily seen to be relatively compact in $L^1$, as we see below, from which we can conclude the relative compactness of solutions of
$(P')$.
\end{remark}

\begin{lemma} \label{bounded:sds:u^sharp}
Let $u_0 \in L^\infty(\Omega)$. Then
the solution orbit $\{ u^\sharp(t): t \ge 0 \}$  is relatively compact in $L^1(\Omega^\sharp)$.
\end{lemma}

Before proving Lemma \ref{bounded:sds:u^sharp},
let us  recall the definition of the BV-norm of functions of one variable (cf. \cite{chap2ambrosio}).

\begin{definition}\label{dinh,ghia:27:10:}
 Let $-\infty \le a<b \le \infty$ and let $w \in L^1(a,b)$. The BV-norm of $w$ is defined by
$$\|w\|_{BV(a, b)}:=\|w\|_{L^1(a, b)}+\ess V_a^b w,$$
where
$$\ess V_a^b w:=\inf \left\{  pV_a^b \widetilde w: \widetilde w = w\mbox{~~a.e. in~~}\Omega \right \}.$$
Here
$$pV_a^b \widetilde w:=\sup \left\{  \sum_{j=1}^{n}  |\widetilde w(t_{j+1})-\widetilde w(t_j) | \right\},$$
where the supremum is taken over all finite partitions ${a < t_1 < \cdots < t_{n+ 1} < b}$.
\end{definition}

\begin{proof}[\bf Proof of Lemma \ref{bounded:sds:u^sharp}]
Let $s_1, s_2$ be as in $\bf (H)$.
Since $u^\sharp(y,t)$ is nonincreasing in $y$ and
$$s_1 \le u^\sharp(y,t) \le s_2  \mbox{~~for all~~} y \in \Omega^\sharp=(0, |\Omega|), t \ge 0,$$
there exists a constant $c>0$ such that
$$\|u^\sharp(t)\|_{BV(\Omega^\sharp)} \le c  \mbox{~~for all~~} t \ge 0.$$
Because of the compactness of the embedding
$BV(\Omega^\sharp) \hookrightarrow L^1(\Omega^\sharp)$,
the set $\{u^\sharp(t): t \ge 0\}$ is relatively compact in
$L^1(\Omega^\sharp)$.
This completes the proof of Lemma \ref{bounded:sds:u^sharp}.
\end{proof}


 \subsection{Proof of Theorem \ref{principal:theorem:to:kai:po}\,(i)}
\label{moiquanhe:sectionP-P:sharp}

\begin{lemma} \label{chuyenthanhbode:quenheeeee:lemmme:u-shapr:u}
 Let $u$ be the solution of $(P')$ with $u_0 \in L^\infty(\Omega)$ and let $u^\sharp$ be as in
 \eqref{defition:u:sharp}. Then
 \begin{align}\label{danthuc123:u-sparp-u:13:3:14}
\|u^\sharp(t)-u^\sharp(\tau)\|_{L^1(\Omega^\sharp)}=\|u(t)-u(\tau)\|_{L^1(\Omega)},
\end{align}
for any $t , \tau \in [0, \infty)$.
\end{lemma}

\begin{remark}
The inequality $\|u^\sharp(t)-u^\sharp(\tau)\|_{L^1(\Omega^\sharp)} \le \|u(t)-u(\tau)\|_{L^1(\Omega)}$ is clear from the general property (iv) of Proposition \ref{chap2:mancyuon:ws:ss}. The important point of the lemma is that equality holds in \eqref{danthuc123:u-sparp-u:13:3:14}.
 \end{remark}

 \begin{proof}[\bf Proof of Lemma \ref{chuyenthanhbode:quenheeeee:lemmme:u-shapr:u}]
Applying Proposition \ref{chap2:mancyuon:ws:ss}\,(iii) for $G(s):=|Y(t;s)-Y(\tau;s)|$ with $t, \tau \ge 0$, we obtain
\begin{align*}
\int_{\Omega^\sharp} |Y(t;u_0^\sharp) -Y(\tau;u_0^\sharp)| =\int_{\Omega} |Y(t;u_0) -Y(\tau;u_0)|.
\end{align*}
This, together with  \eqref{problem:ODE:22:6:3:14} and \eqref{chap3:6:09:bdt:cho:nghiem:u:sao:ha:ha}, yields
$$\|u^\sharp(t)-u^\sharp(\tau)\|_{L^1(\Omega^\sharp)}=\|u(t)-u(\tau)\|_{L^1(\Omega)}.$$
\end{proof}

\begin{corollary} \label{quenheeeee:lemmme:u-shapr:u}
Let $\{t_n\}$ be a sequence of positive numbers such that $t_n \to \infty$ as $n \to \infty$. Then
 the following statements are equivalent
\begin{enumerate}[label=\emph{(\alph*)}]
\item $u^\sharp(t_n) \to \psi$ in $L^1(\Omega^\sharp)$ as $n \to \infty$
for some $\psi\in L^1(\Omega^\sharp)$;
\item $u(t_n) \to \varphi$ in $L^1(\Omega)$ as $n \to \infty$
for some $\varphi\in L^1(\Omega)$ with $\varphi^\sharp=\psi$.
\end{enumerate}
\end{corollary}

\begin{proof}
By \eqref{danthuc123:u-sparp-u:13:3:14}, $u^\sharp(t_n)$ is a Cauchy
sequence in $L^1(\Omega^\sharp)$ if and only if $u(t_n)$ is a Cauchy
sequence in $L^1(\Omega)$.  Therefore the convergence of the
two sequences are equivalent. The fact that $\varphi^\sharp=\psi$
follows also from \eqref{danthuc123:u-sparp-u:13:3:14} (or from
Proposition \ref{chap2:mancyuon:ws:ss}\,(iv)), since $u(t_n) \to \varphi$ in
$L^1(\Omega)$ implies $u^\sharp(t_n) \to \varphi^\sharp$ in
$L^1(\Omega^\sharp)$.
\end{proof}

The following is an immediate consequence of Lemma
\ref{bounded:sds:u^sharp} and Corollary
\ref{quenheeeee:lemmme:u-shapr:u}.

\begin{proposition}\label{boundedness:solution:P:15}
Assume that $u_0 \in L^\infty(\Omega)$ and let $u(t)$ be the
solution of $(P')$.
Then $\{ u(t): t \ge 0 \}$  is relatively compact in $L^1(\Omega)$.
\end{proposition}

\begin{proof}[\bf Proof of Theorem \ref{principal:theorem:to:kai:po}\,(i)]
Since the solution orbit $\{ u(t): t \ge 0 \}$ is relatively compact in $L^1(\Omega)$,
the assertion of Theorem \ref{principal:theorem:to:kai:po}\,(i)
follows from the standard theory of dynamical systems.
\end{proof}


\section{Proof of Theorems \ref{principal:theorem:to:kai:po}\,(ii) and \ref{theorem:multi-stable:case}}
\label{section4ch}

\begin{lemma}[Lyapunov functional]\label{chap3:3chapchapLyapunov-functional}
Assume that $u_0 \in L^\infty(\Omega)$ and let $E$ be the functional defined in \eqref{3:11:2014:Hamlyapunove}. Then the following holds:
\begin{enumerate}[label=\emph{(\roman*)}]
\item For all $\tau_2 > \tau_1 \ge 0$,
\begin{align*}
E(u(\tau_2))-E(u(\tau_1))= -\int_{\tau_1}^{\tau_2}\int_{\Omega} |u_t|^2\,dxdt \le 0.
\end{align*}
\item $E(u(t))$ is nonincreasing on $[0, \infty)$ and has
a limit as $t \to \infty$.
\item $E$ is constant on $\omega(u_0)$.
\end{enumerate}
\end{lemma}

\begin{proof}
(i) We have
\begin{align*}
\frac{d}{dt} E(u(t))&=-\frac{d}{dt} \int_{\Omega} F(u)\,dx
= -\int_{\Omega} f(u)u_t\,dx.
\end{align*}
Since
$$\int_{\Omega} u_t \,dx=0,$$
it follows that
\begin{align*}
\frac{d}{dt} E(u(t)) = -\int_{\Omega} (f(u)-\langle f(u) \rangle) u_t \,dx=-\int_{\Omega} |u_t|^2\,dx.
\end{align*}
Integrating this identity from $\tau_1$ to $\tau_2$, we obtain
$$E(u(\tau_2))-E(u(\tau_1))= -\int_{\tau_1}^{\tau_2}\int_{\Omega} |u_t|^2 \,dxdt\le 0.$$

(ii) By (i), $E(u(\cdot))$ is nonincreasing. Furthermore, it is clearly bounded. Hence it has a limit as $t \to \infty$.

(iii) Let $\varphi \in \omega(u_0)$ and let $t_n\to \infty$ be a sequence such that
\begin{equation*}\label{chap3:3chapchapsuhoitu:point:omega:limit:set}
 u(t_n) \to \varphi\mbox{~~in~~}L^1(\Omega) \mbox{~~as~~} n\to\infty.
\end{equation*}
Then, since $\{u(x,t_n)\}$ is uniformly bounded, the convergence $u(t_n) \to \varphi$ implies
$E(u(t_n)) \to E(\varphi)$ as $n \to \infty$. Consequently,
$$E(\varphi)=\lim_{n \to \infty} E(u(t_n))=\lim_{t \to \infty} E(u(t)).$$
Therefore $E$ is constant on $\omega(u_0)$.
\end{proof}

\begin{remark}\label{gradientlfow:chobaitaon}
If $f$ is globally Lipschitz continuous on $\R$,
Problem $(P')$ is well-posed in $L^2(\Omega)$. It is then easily seen that $(P')$ represents a gradient flow for $E(u)$ on the space
 $$X=\{w \in L^2(\Omega): \int_\Omega w(x) \,dx=\int_\Omega u_0(x) \,dx\}.$$
\end{remark}

\begin{corollary}\label{hequa:ham:lyapunov}
Let $\varphi \in \omega(u_0)$, $\{t_n\}$ be such that
\begin{equation}\label{dayso:hoitu:cho:omega:limit}
t_n \to \infty, \quad u(t_n) \to \varphi \mbox{~~in~~}L^1(\Omega) \mbox{~~as~~} n \to \infty.
\end{equation}¥
 Let $T>0$, $\eta>0$. Then there exists $n_0=n_0(T,\eta)$ such that
\begin{equation}\label{bdt:quatr:ldej^wmksdfsdf}
\|u(t)-\varphi\|_{L^1(\Omega)} \le \eta \mbox{~~for all~~} t \in [t_n, t_n+T], n \ge n_0.
\end{equation}
\end{corollary}

\begin{proof} It follows from Lemma
\ref{chap3:3chapchapLyapunov-functional} that
$$\int_0^{\infty} \int_\Omega |u_t|^2\,dxdt=E(u_0)-
\lim_{t \to \infty} E(u(t)) <\infty.$$
Thus
\begin{equation}\label{dk:tinhchat:sdfse,obbbjf}
\int_{t_n}^{t_n+T} \int_\Omega |u_t|^2 \,dxdt \to 0 \mbox{~~as~~}n \to \infty.
\end{equation}
By the Schwarz inequality, we have, for all $t \in [t_n, t_n+T]$,
\begin{align*}
\|u(t)-\varphi\|_{L^1(\Omega)} & \le \|u(t)-u(t_n)\|_{L^1(\Omega)}+\|u(t_n)-\varphi\|_{L^1(\Omega)}\\
&\le \int_{t_n}^t \|u_t(\tau)\|_{L^1(\Omega)} \,d\tau+ \|u(t_n)-\varphi\|_{L^1(\Omega)}\\
& \le \int_{t_n}^{t_n+T} \|u_t(\tau)\|_{L^1(\Omega)} d\tau+ \|u(t_n)-\varphi\|_{L^1(\Omega)}\\
& \le T^{\frac{1}{2}} |\Omega|^{\frac{1}{2}} \left (\int_{t_n}^{t_n+T} \int_\Omega |u_t(s)|^2\,dxd\tau\right )^{\frac{1}{2}}+\|u(t_n)-\varphi\|_{L^1(\Omega)} \to 0,
\end{align*}
as $n \to \infty.$
This completes the proof of Corollary \ref{hequa:ham:lyapunov}.
\end{proof}

\begin{proof}[\bf Proof of Theorem \ref{principal:theorem:to:kai:po}\,(ii)] Let $\varphi \in \omega(u_0)$ and let $\{t_n\}$ be a sequence as in \eqref{dayso:hoitu:cho:omega:limit}.
Since $s_1 \le u(t_n) \le s_2$ for all $n \ge 0$, it follows that
$$s_1 \le \varphi \le s_2 \mbox{~~a.e. in~~}\Omega.$$
Next we show that $\varphi$ is a stationary solution of $(P')$.
It follows from Corollary \ref{hequa:ham:lyapunov} that for all
$\tau\ge 0$,
$$u(t_n+\tau) \to \varphi\mbox{~~in~~}L^1(\Omega) \mbox{~~as~~} n
\to\infty.$$
Therefore the uniform boundedness of $u$ in $\Omega \times (0, \infty)$ and the Lipschitz continuity of $f$ imply that for all $\tau\ge 0$,
$$f(u(t_n+\tau))-\langle f(u(t_n+\tau)) \rangle  \to f(\varphi)-\langle f(\varphi) \rangle \mbox{~~in~~}L^2(\Omega) \mbox{~~as~~}n \to\infty.$$
Hence
$$\int_\Omega |u_t(t_n+\tau)|^2\,dx \to \int_\Omega  [f(\varphi)-\langle f(\varphi) \rangle]^2 \,dx \mbox{~~as~~}n \to\infty.$$
This and \eqref{dk:tinhchat:sdfse,obbbjf} imply
$$f(\varphi)=\langle f(\varphi) \rangle \mbox{~~a.e. in~~}\Omega.$$
Thus $\varphi$ is a stationary solution of Problem $(P')$.
\end{proof}

\begin{corollary}\label{cor:omega-charact}
Let $u_0 \in L^\infty(\Omega)$ and let $\varphi \in \omega(u_0)$. Then up to modification on a set of zero measure, $\varphi$ is a step function.
More precisely, the following hold:
\begin{enumerate}[label=\emph{(\roman*)}]
\item If $\langle f(\varphi) \rangle  \not\in[f(m), f(M)]$, then
$$\varphi(x) = \langle u_0 \rangle \mbox{~~for a.e.~~}x\in\Omega.$$
\item  If $\langle f(\varphi) \rangle \in (f(m), f(M))$, then
$$\varphi=a_- \chi_{A_-}+a_0 \chi_{A_0}+a_+ \chi_{A_+},$$
where
$s_*<a_- < m ,a_0 \in (m, M),M<a_+<s^*$ satisfy
$$f(a_-)=f(a_0)=f(a_+)=\langle f(\varphi) \rangle;$$
and $A_-, A_0 ,A_+$ are pairwise disjoint subsets of $\Omega$ such that
$$A_- \cup A_0 \cup A_+=\Omega.$$
\item If $\langle f(\varphi) \rangle=f(m)$, then
$$\varphi=m \chi_{A_-}+s^*\chi_{A_+},$$
 where $A_-$ and $A_+$ are disjoint and $A_- \cup A_+=\Omega.$
\item If $\langle f(\varphi) \rangle=f(M)$, then
$$\varphi=s_* \chi_{A_-}+M\chi_{A_+},$$
 where $A_-$ and $A_+$ are disjoint and $A_- \cup A_+=\Omega$.
\end{enumerate}
\end{corollary}

\begin{proof}
(i) Since $\langle f(\varphi) \rangle \not\in [f(m), f(M)]$, the equation $f(s)= \langle f(\varphi) \rangle$ has a unique solution.
Therefore $\varphi(x) $ is constant on $\Omega$ (in the sense of almost everywhere). By the mass conservation property, we have
$$\varphi(x)=\langle  \varphi \rangle =\langle  u_0 \rangle \mbox{~~for a.e.~~}x\in\Omega.$$

(ii) If $\langle f(\varphi) \rangle \in (f(m), f(M))$, then the equation $f(s)=\langle f(\varphi) \rangle$ has exactly three solutions which are denoted by $a_-, a_0, a_+$ with
$a_-<a_0<a_+.$
Hence $\varphi$ takes at most three values $a_-, a_0, a_+$. This proves (ii).

(iii), (iv) The proof (iii) and (iv) are similar to that of (i) and (ii). We omit them.
\end{proof}

\begin{proof}[\bf Proof of Theorem \ref{theorem:multi-stable:case}]
The inequality $a_n\leq u(x,t)\leq b_n$ follows from Lemma 
\ref{lem:ab-bound}. Once this inequality is shown, the rest of 
the proof of Theorem \ref{theorem:multi-stable:case} is virtually 
the same as that of Theorems \ref{principal:theorem} and 
\ref{principal:theorem:to:kai:po}. 
\end{proof}


\section{Proof of Theorem \ref{principal:theorem2}}
\label{section:principal123:theorem2}

In this section we discuss the large time behavior of the solution
under the assumption $\langle u_0\rangle\not\in[s_*,s^*]$ and
prove Theorem \ref{principal:theorem2}.
For later convenience, we introduce the differential operator
$\mathcal L$:
\begin{equation}\label{definition:operator:L}
\mathcal L(Z(t))=\dot Z(t)-f(Z(t)) +\lambda(t),
\end{equation}
where $\lambda(t)$ is as in \eqref{definition:lamda:t}.
The following comparison principle is quite standard in the ODE
theory; see, e.g., \cite[Theorem 6.1, page 31]{chap2hale}.

\begin{proposition}\label{theorem-dinhly-sosanh}
Let $T > 0$ and let $Z_1, Z_2 \in C^1([0,T])$ satisfy
\begin{equation*}
\begin{cases}
\mathcal L (Z_1(t)) \le \mathcal L (Z_2(t)) \mbox{~~for all~~} t \in [0, T],\vspace{6pt}\\
Z_1(0) \le Z_2(0).
\end{cases}
\end{equation*}
Then
$$Z_1(t) \le Z_2(t) \mbox{~~for all~~} t \in [0, T].$$
\end{proposition}

The following lemma will play an important role in this and the
next section.
\begin{lemma}\label{lemma:test:thunghiem}
Let $u$ be the solution of Problem $(P')$ with $u_0 \in L^\infty(\Omega)$ and let $\lambda$ be as defined in \eqref{definition:lamda:t}.
Let $\delta>0,$  $k \in \R$ be constants satisfying
$$\inf_{s \in \R} f(s)<k-\delta<k+\delta<\sup_{s \in \R} f(s),$$
and set
$$\alpha_-=\min\{s \in \R: f(s)=k+\delta\}, \quad \alpha_+=\max\{s \in \R: f(s)=k-\delta\}.$$
Let $s_1, s_2$ be as in $\bf (H)$.
Then, for each $\e>0$, there exists $T_0=T_0(\delta,\e,k,s_1,s_2)>0$ such that if
\begin{equation} \label{new:===giatthiet:alpha:to:2}
 |\lambda(t)-k| \le \delta \mbox{~~for all~~} t \in [\tau, \tau+T_0], \mbox{~~with some~~} \tau \ge 0,
\end{equation}
then
\begin{equation*}\label{ketluan:alpha:to:3:iyuggjt9800}
\alpha_- -\e \le Y(\tau+T_0; s) \le \alpha_+ + \e \mbox{~~for all~~}s\in[s_1, s_2].
\end{equation*}
\end{lemma}

\begin{figure}[!h]
\begin{center}
\includegraphics[scale=0.5]{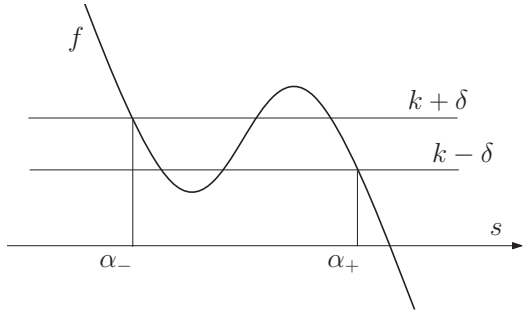}
\caption{A typical image of $\alpha_-, \alpha_+$}
\label{default}
\end{center}
\end{figure}

\begin{proof}
Let $Z_1, Z_2 \in C^1([0, \infty))$ be the unique solutions of the
problems
\begin{equation*}
 \ \left\{
\begin{array}{ll}
\dot{Z_1}=f(Z_1)- k-\delta, \quad t > 0, \vspace{6pt} \\
Z_1(0)=s_1,
\end{array}
\right.  \quad \mbox{and} \quad
\ \left\{
\begin{array}{ll}
\dot{Z_2}=f(Z_2)- k+\delta, \quad t > 0, \vspace{6pt} \\
Z_2(0)=s_2.
\end{array}
\right.
\end{equation*}
It is easy to see that
$\lim_{t \to \infty} Z_1(t) \ge \alpha_-$ and that $\lim_{t \to \infty} Z_2(t) \le \alpha_+$. Thus there exists $T_0>0$ such that
\begin{equation}\label{bdt:Z-1=Z-2=213:to:4}
Z_1(T_0)\ge \alpha_--\e, \quad Z_2(T_0)\le \alpha_++\e.
\end{equation}
Now assume that \eqref{new:===giatthiet:alpha:to:2} holds for the above $T_0$ and some $\tau \ge 0$.
Set
$$\widetilde Z_1(t):=Z_1(t-\tau), \quad  \widetilde Z_2(t):=Z_2(t-\tau) \mbox{~~for~~} t \ge \tau.$$
Then, by Corollary \ref{he qua:lalllaa:dfes}, we have
$$\widetilde Z_1(\tau)=s_1 \le Y(\tau;s) \le s_2=\widetilde Z_2(\tau) \mbox{~~for all~~} s \in [s_1, s_2].$$
Let $\mathcal L$ be as defined in \eqref{definition:operator:L}. It follows from \eqref{new:===giatthiet:alpha:to:2} that
$$\mathcal L (\widetilde Z_1(t)) \le \mathcal L(Y(t;s))=0 \le \mathcal L (\widetilde Z_2(t)) \mbox{~~for all~~} t \in [\tau, \tau+T_0], s \in [s_1, s_2].$$
Hence, by Proposition \ref{theorem-dinhly-sosanh}, we have
$$\widetilde Z_1(t) \le Y(t;s) \le \widetilde Z_2(t) \mbox{~~for all~~} t \in [\tau, \tau+T_0], s \in [s_1, s_2].$$
Setting $t= \tau+T_0$ and recalling \eqref{bdt:Z-1=Z-2=213:to:4}, we obtain
$$\alpha_- -\e \le Y(\tau+T_0; s) \le \alpha_+ + \e \mbox{~~for all~~}s\in[s_1, s_2].$$
The lemma is proved.
\end{proof}

\begin{lemma}\label{to100:bodecho1.4}
Let $u$ be the solution of Problem $(P')$ with $u_0 \in L^\infty(\Omega)$.
Assume furthermore that $\langle u_0 \rangle \not \in [s_*, s^*]$. Then
\begin{enumerate}[label=\emph{(\roman*)}]
\item $\omega(u_0) =\{\langle u_0 \rangle\}$,
\item if $\langle u_0 \rangle < s_*$ (resp. $\langle u_0 \rangle > s^*$),
then there exists $t_1 \ge 0$ such that for all $t \ge t_1$,
$$u(x,t)  \le  s_* \mbox{~~(resp.~~}\ge  s^*)\mbox{~~for a.e.~~} x\in \Omega.$$
\end{enumerate}

\begin{proof}
(i) We first prove that
\begin{equation}\label{to10:cmrang:varphi:outsides*s*}
\langle f(\varphi) \rangle \not\in [f(m), f(M)] \mbox{~~for all~~}\varphi \in \omega(u_0).
\end{equation}
Suppose, to the contrary, that there exists  $\varphi \in \omega(u_0)$ such that $\langle f(\varphi) \rangle \in [f(m), f(M)].$
 Then, in view of Corollary \ref{cor:omega-charact}\,(ii), (iii) and (iv), we have
$$s_* \le \varphi(x) \le s^*\mbox{~~for a.e.~~}x \in \Omega.$$
This together with the mass conservation yields
$$\langle u_0 \rangle =\langle \varphi \rangle \in [s_*, s^*],$$
which contradicts the hypothesis of the lemma. Thus \eqref{to10:cmrang:varphi:outsides*s*} holds.
Consequently,  by Corollary \ref{cor:omega-charact}\,(i),  any element $\varphi \in \omega(u_0)$ satisfies
$$\varphi(x)=\langle  u_0 \rangle \mbox{~~for a.e.~~}x\in\Omega.$$
Therefore, $\omega(u_0)$ has the unique element $\langle u_0 \rangle$.

(ii) We only consider the case where $\langle u_0 \rangle < s_*$.
The proof of the other case is similar and we omit it. Set
$$m_0 :=\langle u_0 \rangle <s_*.$$
Choose $\delta>0$ such that $f(m_0)-\delta>f(s_*)$. Then the equation $f(s)=f(m_0)+\delta$ (resp. $f(s)=f(m_0)-\delta$) has a unique root, which we denote by $\alpha_-$ (resp. $\alpha_+$).
Clearly we have
$$\alpha_-< m_0<\alpha_+<s_*.$$
 Choose $\e>0$ such that
$$\alpha_-- \e <\alpha_++\e <s_*.$$
It follows from (i) that $u(t) \to m_0$ in $L^1(\Omega)$ as $t \to \infty$. Hence
$\lambda(t) \to f(m_0)$ as $t \to \infty$. Consequently, there exits $\tau \ge 0$ such that
\begin{equation}\label{estimate:lmabda:hseto90}
|\lambda(t) -m_0| \le \delta \mbox{~~for all~~} t \ge \tau.
\end{equation}
Let $s_1, s_2$ be as in $\bf (H)$.
Then it follows from \eqref{estimate:lmabda:hseto90} and Lemma \ref{lemma:test:thunghiem} that there exists $T_0>0$ such that
\begin{align*}
\alpha_- -\e \le Y(\tau+T_0,s) \le \alpha_+ + \e \mbox{~~for all~~}s\in[s_1, s_2].
\end{align*}
This, together with \eqref{problem:ODE:22:6:3:14}, implies
$$\alpha_-  -\e \le u(x,\tau+T_0) \le \alpha_+ + \e\mbox{~~for a.e.~~}x\in\Omega.$$
In particularly, $u(x,\tau+T_0) \le s_*\mbox{~~for a.e.~~}x\in\Omega.$
Hence, by Corollary \ref{moiquanheproblemP-ODE}\,(ii), we have for all $t \ge \tau+T_0$,
$$u(x,t) \le s_* \mbox{~~for a.e.~~}x\in\Omega.$$
This completes the proof of Lemma \ref{to100:bodecho1.4}.
\end{proof}

\end{lemma}

\begin{proof}[\bf Proof of Theorem \ref{principal:theorem2}]
It follows from Lemma \ref{to100:bodecho1.4}\,(i) that $\omega(u_0)=\{\langle u_0 \rangle\}.$  Hence $u(t) \to \langle u_0 \rangle$ in $L^1(\Omega)$ as $t \to \infty$.
We will show that the convergence takes place in $L^\infty(\Omega)$ and that the estimate \eqref{convergecne:rate:to:10000} holds. We will only consider the case $\langle u_0\rangle < s_*$, since the case $\langle u_0\rangle> s^*$ can be treated similarly.  Then, in view of Lemma \ref{to100:bodecho1.4}\,(ii), we may assume without loss of generality that \begin{equation}\label{giathietchungminh12chodieukiensau1khoangthoigian}
u(x, t)  \le  s_* \ \ \mbox{for \ a.e.}\ x\in \Omega, \ t\geq 0.
\end{equation}
Thus we may choose $s_1, s_2$ such that
\[
s_1<\langle u_0\rangle <s_2 \leq s_*.
\]
Since $u(t)\to \langle u_0\rangle$ in $L^1(\Omega)$ as $t\to\infty$, we have $\lambda(t):=\langle f(u(t))\rangle\to f(\langle u_0\rangle)$ as $t\to\infty$. Hence \begin{equation}\label{lambda-s1s2}
f(s_1) \geq \lambda(t)\geq f(s_2)
\end{equation}
for all large $t\geq 0$. Without loss of generality we may assume that \eqref{lambda-s1s2} holds for all $t\geq 0$. Now we set $v_1(t):=Y(t;s_1),\, v_2(t):=Y(t;s_2)$. Then
\[
\dot v_1 =f(v_1)-\lambda(t),\ \ v_1(0)=s_1,\quad\ \ \dot v_2 =f(v_2)-\lambda(t),\ \ v_2(0)=s_2.
\]
Therefore, by \eqref{lambda-s1s2},
\[
{\mathcal L}(s_1)\leq 0={\mathcal L}(v_1)={\mathcal L}(v_2)
\leq {\mathcal L}(s_2).
\]
Hence, by Proposition \ref{theorem-dinhly-sosanh} and the
monotonicity of $Y(t;s)$ in $s$, we have
\begin{equation}\label{order5}
s_1\leq v_1(t)\leq u(x,t)\leq v_2(t) \leq s_2 \quad\ \hbox{for \ \ a.e.}\ x\in\Omega,\ t\geq 0.
\end{equation}
Define $\rho(t):=v_2(t)-v_1(t)$. Then
\[
\dot \rho=f(v_2)-f(v_1)=f'(\xi(t))\rho,
\]
where $\xi(t)\in[v_1(t),v_2(t)]\subset[s_1,s_2]$.  Hence, by \eqref{order5},
\[
\dot \rho(t)\leq -\mu\rho(t),\quad\hbox{where}\ \ \ -\mu:=\max_{s\in[s_1,s_2]} f'(s)<0.
\]
Consequently
\[
v_2(t)-v_1(t):=\rho(t) = O(e^{-\mu t}).
\]
Combining this and \eqref{order5}, we obtain
\[
\Vert u(t)-\langle u_0\rangle\Vert_{L^\infty(\Omega)}=
O(e^{-\mu t}).
\]
This completes the proof of Theorem \ref{principal:theorem2}.
\end{proof}


\section{Proof of Theorem \ref{hequa:chinh4}}
\label{section:chungminh:menhde:coban}

In this section we discuss the large time behavior of the solution
under the assumption $s_* \le \langle u_0 \rangle \le s^*$ and
prove Theorem \ref{hequa:chinh4}.
The proof is divided into two steps. Subsection
\ref{sec:for:prove:bistable:case:u0instar} deals with the
special case where $s_* \le u_0(x) \le s^*$ a.e. $x\in\Omega$,
and Subsection \ref{sec:for:prove:bistable:case:u0outsidestar} deals
with the general case.


\subsection{The case $s_* \le u_0 \le s^*$ a.e. in $\Omega$}
\label{sec:for:prove:bistable:case:u0instar}

The main result of this subsection is Theorem \ref{principal:theorem3} below.
First note that for each $u_0 \in L^\infty(\Omega)$ with
$s_* \le u_0 \le s^*$ a.e. in $\Omega$ there exist subsets
$\widetilde \Omega, \mathcal N\subset\Omega$ such that
\begin{align*}
&|\mathcal N|=0, \quad \widetilde \Omega \cup \mathcal N=\Omega,\\
&s_* \le u_0(x) \le s^* \mbox{~~for all~~} x \in \widetilde \Omega.
\end{align*}
We define for each $t\ge 0$,
\begin{align*}
&\Omega_-(t):=\{x \in \widetilde \Omega: Y(t;u_0(x)) \le m\},\\
&\Omega_0(t):=\{x \in \widetilde \Omega: m < Y(t;u_0(x)) <  M\},\\
&\Omega_+(t):=\{x \in \widetilde \Omega: Y(t;u_0(x)) \ge M\}.
\end{align*}

\begin{lemma}\label{dondieu:tap:unstabel}
Suppose that
$$s_* \le u_0 \le s^*\mbox{~~a.e. in~~}\Omega.$$
 Then for every $t' >t \ge 0$, $s \in [s_*, s^*]$,
\begin{enumerate}[label=\emph{(\roman*)}]
\item if $Y(t;s) \le m$, then $Y(t';s) \le m$ for every $t' >t \ge 0$;
\item if $Y(t;s) \ge M$, then $Y(t';s) \ge M$ for every $t' >t \ge 0$.
\end{enumerate}
As a consequence, for every $t' > t \ge 0$,
$$\Omega_-(t) \subset \Omega_-(t'), ~~~~ \Omega_+(t) \subset \Omega_+(t') \mbox{~~~~and~~~~} \Omega_0(t) \supset \Omega_0(t').$$
In other words, $\Omega_-(t), \Omega_+(t)$ are monotonically expanding in $t$ while $\Omega_0(t)$ is monotonically shrinking in $t$.
\end{lemma}

\begin{proof}
It follows from Corollary \ref{moiquanheproblemP-ODE}\,(i) that
for all $t \ge 0$,
$$s_* \le u(x,t) \le s^* \mbox{~~for a.e.~~}x\in \Omega,$$
so that
\begin{equation*}\label{tinhchat:lambda:t:sao:s}
f(m) \le \langle f(u(t)) \rangle = \lambda(t) \le f(M).
\end{equation*}
Consequently
\begin{equation*}\label{11bdthucviphanorsay187}
\mathcal {L}(m) \ge 0 \mbox{~~and~~}\mathcal {L}(M) \le 0,
\end{equation*}
where $\mathcal L$ is as defined in \eqref{definition:operator:L}.
Hence (i), (ii) follow from Proposision \ref{theorem-dinhly-sosanh}.
\end{proof}

We define
\begin{equation}\label{dinhnghia:hop:car:omega,levelset}
\Omega_-(\infty):=\bigcup_{t \ge 0} \Omega_-(t), \quad \Omega_0(\infty):=\bigcap_{t \ge 0} \Omega_0(t), \quad \Omega_+(\infty):=\bigcup_{t \ge 0} \Omega_+(t).
\end{equation}
The following result is an immediate consequence of Lemma
\ref{dondieu:tap:unstabel}.

\begin{corollary} \label{H:Q/W/S:12:09:asfspdfjf123}
Suppose that $s_* \le u_0 \le s^* \mbox{~~a.e. in~~}\Omega$.
Then
$\Omega_-(\infty), \Omega_0(\infty), \Omega_+(\infty)$ are  pairwise disjoint and
$$\Omega=\Omega_-(\infty)\cup \Omega_0(\infty) \cup \Omega_+(\infty) \cup \mathcal N.$$
Moreover,
\begin{align*}
 \Omega_-(\infty)&=\{x\in \widetilde \Omega:  \exists t_0 \ge 0  \mbox{~~such that~~}  Y(t;u_0(x))\le m \mbox{~~for all~~} t \ge t_0\}, \\
\Omega_0(\infty)&=\{x\in \widetilde \Omega:  m<Y(t;u_0(x))<M \mbox{~~for all~~} t \ge 0\},\\
 \Omega_+(\infty)&=\{x\in \widetilde \Omega:  \exists t_0 \ge 0  \mbox{~~such that~~}  Y(t;u_0(x))\ge M \mbox{~~for all~~} t \ge t_0\}.
\end{align*}
\end{corollary}

\begin{proposition} \label{principal:theorem3555}
Assume that
\begin{equation}\label{giathietsdfs:sdiehisehngay258888}
s_* \le u_0 \le s^* \mbox{~~a.e. in~~}\Omega.
\end{equation}
Then
\begin{equation}\label{phan1cuadinhlyquantrongt2dsfsdfsd}
\langle f(\varphi) \rangle \in [f(m), f(M)] \mbox{~~for all~~} \varphi \in \omega(u_0).
\end{equation}
Furthermore,
\begin{enumerate}[label=\emph{(\roman*)}]
\item If $f(m) <\langle f(\varphi) \rangle <f(M)$, then
$$\varphi= a_- \chi_{\Omega_-(\infty)}+a_0 \chi_{\Omega_0(\infty)}+ a_+ \chi_{\Omega_+(\infty)},$$
where $s_*<a_- < a_0 < a_+ <s^*$ are the roots of the equation
$f(s)=\langle f(\varphi) \rangle$.
\item If $\langle f(\varphi) \rangle=f(m)$, then
$$\varphi= m \chi_{\Omega_-(\infty)}+s^* \chi_{\Omega_0(\infty) \cup \Omega_+(\infty)}.$$
\item If $\langle f(\varphi) \rangle=f(M)$, then
$$\varphi= s_* \chi_{\Omega_-(\infty) \cup \Omega_0(\infty)}+M \chi_{ \Omega_+(\infty)}.$$
\end{enumerate}
\end{proposition}

\begin{proof}
First we prove \eqref{phan1cuadinhlyquantrongt2dsfsdfsd}.
Indeed, if there exists $\psi \in \omega(u_0)$ such that
$\langle f(\psi) \rangle \not\in [f(m), f(M)]$, then it follows from
Corollary \ref{cor:omega-charact}\,(i) that
$$\psi \equiv \langle u_0 \rangle \not \in [s_*, s^*],$$
which is in contradiction with \eqref{giathietsdfs:sdiehisehngay258888}.
Thus \eqref{phan1cuadinhlyquantrongt2dsfsdfsd} holds.
Next we remark that there exists a sequence $t_n \to \infty$ such that
$$Y(t_n;u_0(x)) \to \varphi(x) \quad \mbox{for a.e.~} x\in \Omega
\mbox{~~as~~} n \to \infty.$$
This, together with Corollary \ref{H:Q/W/S:12:09:asfspdfjf123}, implies that
\begin{align*}
&\varphi(x) \le m \quad\quad \quad \mbox{~~for a.e.~~} x \in
\Omega_-(\infty),\\
&m \le \varphi(x) \le M \mbox{~~~~for a.e.~~} x \in \Omega_0(\infty),\\
&\varphi(x) \ge M \quad\quad\quad \mbox{~~for a.e.~~} x \in
\Omega_+(\infty).
\end{align*}
Thus we deduce from Corollary \ref{cor:omega-charact} that
\begin{align*}
&\varphi=a_- \mbox{~~a.e. in~~}\Omega_-(\infty), \quad \varphi=a_0 \mbox{~~a.e. in~~}\Omega_0(\infty),\\
&\varphi=a_+\mbox{~~a.e. in~~}\Omega_+(\infty),
\end{align*}
if $\langle f(\varphi) \rangle \in (f(m), f(M))$. This proves (i).
The assertions (ii), (iii) also follow from Corollary \ref{cor:omega-charact}.
The proposition is proved.
\end{proof}

\begin{lemma}\label{keylemma:123}
Assume that $s, \widetilde s$ satisfy
$$m<Y(t; s)<M, \quad m<Y(t; \widetilde s)<M \mbox{~~for all~~} t \ge 0.$$
Assume further that there exists a sequence $t_n \to \infty$ satisfying
\begin{equation}\label{to:300:dkien:kien:hoitu}
Y(t_n; s) \to a, \quad Y(t_n; \widetilde s) \to a \mbox{~~as~~}n \to \infty,
\end{equation}
for some $a \in [m, M]$.
Then $s=\widetilde s.$
\end{lemma}

\begin{proof}
Assume that $s \le \widetilde s$. Then
$$m<Y(t; s) \le Y(t; \widetilde s) <M\mbox{~~for all~~}t \ge 0.$$
Since $f$ is increasing on $[m, M]$, we have
$$f(Y(t; s)) \le f(Y(t; \widetilde s))\mbox{~~for all~~}t \ge 0.$$
Consequently, the function
$h(t):= Y(t; \widetilde s)-Y(t; s)$ satisfies
$$h'(t)=\dot{Y}(t; \widetilde s)-\dot{Y}(t; s)=f(Y(t,\widetilde s))-f(Y(t,s)) \ge 0.$$
Therefore $h(t) \ge h(0)=\widetilde s-s$.
On the other hand, \eqref{to:300:dkien:kien:hoitu} yields
$\lim_{n \to \infty}h(t_n)=0$.
Hence $s=\widetilde s$.
\end{proof}

\begin{theorem} \label{principal:theorem3}
Assume that
$s_* \le u_0 \le s^*$ a.e. in $\Omega$. Suppose further that
$$|\{x: u_0(x)=s \}|=0 \mbox{~~for all~~}s \in (m, M).$$
Then $\omega(u_0)$ contains only one element $\varphi$, and it is a step function. Moreover, $\varphi$ is given by
\begin{equation}\label{eq:form:of:varphi:when:omega-0=0}
\varphi=a_- \chi_{\Omega_-(\infty)}+a_+ \chi_{\Omega_+(\infty)},
\end{equation}
where $\Omega_-(\infty), \Omega_+(\infty)$ are defined by \eqref{dinhnghia:hop:car:omega,levelset} and
$a_+, a_-$ are constants satisfying
\begin{align*}
s_* \le a_-\le m, \quad M \le a_+ \le s^*, \quad f(a_-)=f(a_+)=\langle f(\varphi) \rangle.
\end{align*}
\end{theorem}

\begin{proof}
We first show that
$$|\Omega_0(\infty)| = 0.$$
Suppose, by contradiction, that
$$|\Omega_0(\infty)| > 0$$
and let $\varphi \in \omega(u_0)$.
Then in view of Proposition \ref{principal:theorem3555},
 $ \varphi$ is constant on $\Omega_0(\infty)$, say,  $\varphi = a \in [m, M]$ a.e. on $\Omega_0(\infty)$.
 Let $t_n \to \infty$ be such that
$$u(t_n) \to \varphi \mbox{~~in~~}L^1(\Omega) \mbox{~~as~~} n \to \infty,$$
or equivalently,
$$Y(t_n;u_0(\cdot)) \to \varphi \mbox{~~in~~}L^1(\Omega) \mbox{~~as~~} n \to \infty.$$
Then there exists a subsequence $\{t_{n_k}\}$ such that
$$Y(t_{n_k};u_0(x)) \to a  \mbox{~~as~~} k \to \infty \mbox{~~a.e.~~}x\in\Omega_0(\infty).$$
We deduce from Lemma \ref{keylemma:123} that there exists $\alpha \in (m, M)$ such that
$$u_0=\alpha \mbox{~~on~~} \Omega_0(\infty),$$
which implies by the hypothesis in Theorem \ref{principal:theorem3} that
$$|\Omega_0(\infty)|=0.$$
This and Proposition \ref{principal:theorem3555} imply that any element $\varphi \in \omega(u_0)$ has the
form \eqref{eq:form:of:varphi:when:omega-0=0}.

Next we prove that $\omega(u_0)$ contains only one element.
We argue by contradiction, suppose that there exist
$\varphi, \overline \varphi \in \omega(u_0)$ with
$\varphi \neq \overline \varphi$.
By what we have shown above, we can write:
\begin{equation*}
\varphi=a_- \chi_{\Omega_-(\infty)}+a_+ \chi_{\Omega_+(\infty)}, \quad \overline \varphi=\overline a_- \chi_{\Omega_-(\infty)}+\overline a_+ \chi_{\Omega_+(\infty)},
\end{equation*}
where the constants $a_{\pm}, \overline a_{\pm}$ satisfy
\begin{align*}
&a_-, \overline a_- \in [s_*, m], \quad a_+, \overline a_+ \in [M, s^*],\\
&f(a_-)=f(a_+)=\langle f(\varphi) \rangle,\quad f(\overline a_-)=f(\overline a_+)=\langle f(\overline \varphi) \rangle.
\end{align*}
Since $\varphi \neq \overline \varphi$, we have $(a_-, a_+) \neq (\overline a_-, \overline a_+)$, which implies
$\langle f(\varphi) \rangle \neq \langle f(\overline\varphi) \rangle$.
Without loss of generality, we may assume that
$$f(m) \le \langle f(\varphi) \rangle < \langle f(\overline \varphi) \rangle \le f(M).$$
Since $f$ is strictly decreasing on $(-\infty,m]$ and $[M,\infty)$, we see that
$a_- >\overline a_-, a_+ >\overline a_+$. Then,
using the mass conservation property, we obtain
$$\int_\Omega u_0 = a_-|\Omega_-(\infty)|+a_+|\Omega_+(\infty)| > \overline a_-|\Omega_-(\infty)|+ \overline a_+|\Omega_+(\infty)| =\int_\Omega u_0,$$
which gives a contradiction. Hence $\omega(u_0)$ contains only one element. This completes the proof of Theorem \ref{principal:theorem3}.
\end{proof}


\subsection{The general case}
\label{sec:for:prove:bistable:case:u0outsidestar}

In this subsection, we prove Theorem \ref{hequa:chinh4} for the general case.
We begin with the following lemma whose proof is similar to that of Lemma \ref{to100:bodecho1.4}\,(ii).
\begin{lemma}\label{kqso1ddse:12}
Let $u$ be the solution of Problem $(P')$ with $u_0 \in L^\infty(\Omega)$.
If there exists $\varphi \in \omega(u_0)$ such that
$$ \langle f(\varphi)\rangle \in (f(m), f(M)),$$
then there exists $t_1 \ge 0$ such that  for all $t \ge t_1$,
$$ s_* \le u(x,t)  \le s^* \mbox{~~for a.e.~~}x\in\Omega.$$
\end{lemma}

\begin{proof}
Since $\langle f(\varphi)\rangle \in (f(m), f(M))$, the equations
$f(s)=\langle f(\varphi)\rangle$ has the three roots $a_-<a_0<a_+$.
Moreover, we have
$$s_*<a_-,  \quad a_+<s^*.$$
Choose $\delta>0$ such that $f(m)< \langle f(\varphi) \rangle-\delta <\langle f(\varphi) \rangle+\delta <f(M)$ and set
$$\alpha_-:=\min\{s \in \R: f(s)=k+\delta\}, \quad \alpha_+:=\max\{s \in \R: f(s)=k-\delta\}.$$
Clearly, we have
$$s_* <\alpha_-<\alpha_+<s^*.$$
Choose $\e>0$ such that
\begin{equation}\label{one:for:s-star-star:for:dl:epsilon}
s_*<\alpha_-- \e <\alpha_++\e <s^*
\end{equation}
and let $s_1, s_2$ be as in $\bf (H)$.
Let $T_0$ be as in Lemma \ref{lemma:test:thunghiem}. It follows from Corollary \ref{hequa:ham:lyapunov} that
there exits $\tau>0$ such that
$$ |\lambda(t)-\langle f(\varphi) \rangle| \le \delta \mbox{~~for all~~} t \in [\tau, \tau+T_0].$$
Thus, Lemma \ref{lemma:test:thunghiem} yields
$$\alpha_- -\e \le Y(\tau+T_0;s) \le \alpha_+ + \e \mbox{~~for all~~}s\in[s_1, s_2].$$
This, together with \eqref{one:for:s-star-star:for:dl:epsilon} and \eqref{problem:ODE:22:6:3:14}, implies
$$s_* \le u(x,\tau+T_0) \le s^*\mbox{~~for a.e.~~}x\in\Omega.$$
Hence the assertion of Lemma \ref{kqso1ddse:12} follows from Corollary \ref{moiquanheproblemP-ODE}\,(i).
\end{proof}

\begin{proof}[\bf Proof of Theorem \ref{hequa:chinh4}]
First we prove that
\begin{equation}\label{phan1cuadinhlyquantrongt2}
\langle f(\varphi) \rangle \in [f(m), f(M)] \mbox{~~for all~~}
\varphi \in \omega(u_0).
\end{equation}
Indeed, if there exists $\psi \in \omega(u_0)$ such that
$\langle f(\psi) \rangle \not\in [f(m), f(M)]$, then it follows
from Corollary \ref{cor:omega-charact}\,(i)
that $\psi \equiv \langle u_0 \rangle \not \in [s_*, s^*],$
which contradicts the assumption of the theorem. Since $\omega(u_0)$
is connected in $L^1(\Omega)$, the set
$\{\langle f(\varphi)\rangle: \varphi \in \omega(u_0)\}$ is
connected in $\R$. This, together with
\eqref{phan1cuadinhlyquantrongt2}, implies that we have one
and only one of the following cases:
\begin{enumerate}[label=(\roman*)]
\item either there exists $\varphi \in \omega(u_0)$ such that
$\langle f(\varphi)\rangle \in  (f(m), f(M))$,
\item or we have $\langle f(\varphi)\rangle=f(m) \mbox{~~for all~~} \varphi \in \omega(u_0),$
\item or we have $\langle f(\varphi)\rangle=f(M) \mbox{~~for all~~} \varphi \in \omega(u_0).$
\end{enumerate}

If (i) holds, then by Lemma \ref{kqso1ddse:12}, there exists
$t_1 \ge 0$ such that
$$s_* \le u(x, t_1) \le s^* \mbox{~~for a.e~~}x\in\Omega.$$
It is also clear from \eqref{giathietchodinhlyso222dg} and the
strict monotonicity of $s\mapsto Y(t;s)$ that $u(x,t_1):=Y(t_1;u_0(x))$
satisfies
\[
|\{x: u(x,t_1)=s\}|=0 \mbox{~~for all~~}s\in \R.
\]
Hence the assertion of Theorem \ref{hequa:chinh4} follows from
Theorem \ref{principal:theorem3} by regarding $u(x,t_1)$ as
the initial data.

If (ii) holds, then it follows from Corollary
\ref{cor:omega-charact}\,(iii) that for each
$\varphi \in \omega(u_0)$, $\varphi$ is written
--- up to modification on a set of zero measure ---
in the form
\begin{equation}\label{phi}
\varphi(x)=m \chi_{A_-}(x)+s^*\chi_{A_+}(x)
\end{equation}
for some disjoint subsets $A_-, A_+$ satisfying
$A_- \cup A_+=\Omega.$  Now let
\[
{\Omega_+}(t)= \{ x \in\Omega: Y(t;u_0(x)) \ge M \}.
\]
Then, since
$\lambda(t):=\langle f(u(t)) \rangle \to f(m)$ as $t \to \infty,$
there exists $T>0$ such that
\[
\mathcal {L}(M) = -f(M)+\lambda(t) \le 0\quad \
\hbox{for all}\ \ t\geq T,
\]
where $\mathcal L$ is as defined in \eqref{definition:operator:L}.
Therefore, if $Y(t_1;u_0(x))\geq M$ for some $t_1\in[T,\infty)$,
then $Y(t;u_0(x))\geq M$ for all $t\geq t_1$.
This implies that $\Omega_+(t)$ is monotonically expanding for all
$t\geq T$.
Hence the limit $\Omega_+(\infty)$ exists
and we have
\[
\Omega_+(\infty)=\{x\in \Omega:  \exists t_0 \ge 0
\mbox{~~such that~~}  Y(t;u_0(x))\ge M \mbox{~~for all~~} t \ge t_0\}.
\]
It is also
easily seen that the set $A_+$ in \eqref{phi} coincides with this
limit, hence
\[
A_-=\Omega\setminus \Omega_+(\infty), \ A_+=\Omega_+(\infty).
\]
Since any element of $\omega(u_0)$ is written in this form,
$\omega(u_0)$ contains only one element, which establishes
the assertion of theorem.

The case (iii) can be treated by the same argument as (ii). This
completes the proof of Theorem \ref{hequa:chinh4}.
\end{proof}


\end{document}